\documentclass[aap,preprint]{imsart}
\RequirePackage{geometry}
\RequirePackage[OT1]{fontenc}
\RequirePackage{amsmath,amsthm,ifthen,bm}
\RequirePackage{natbib}

\usepackage{amsfonts}
\usepackage[T1]{fontenc}
\usepackage[symbol]{footmisc}
\usepackage{float}
\usepackage{multirow}  
\usepackage{graphicx}
\usepackage{algorithm}
\usepackage{algpseudocode}
\usepackage{enumitem}

\RequirePackage[OT1]{fontenc}
\RequirePackage{amsthm,amsmath}

\theoremstyle{definition}
\newtheorem{theorem}{Theorem}[section]

\newtheorem{definition}[theorem]{Definition}
\newtheorem{lemma}[theorem]{Lemma}
\newtheorem{claim}[theorem]{Claim}
\newtheorem{notation}[theorem]{Notation}
\newtheorem{proposition}[theorem]{Proposition}

\newtheorem{condition}[theorem]{Condition}
\newtheorem{example}[theorem]{Example}

\newtheorem{remark}[theorem]{Remark}

\newcommand{\BE}{\mathsf{E}}
\newcommand{\BP}{\mathsf{P}}

\newcommand{\BR}{\mathbb{R}}

\newcommand{\ve}{\varepsilon}
\newcommand{\Var}{\mathrm{Var}}
\newcommand{\norm}[1]{\left\lVert#1\right\rVert}

\startlocaldefs
\numberwithin{equation}{section}
\theoremstyle{plain}

\endlocaldefs

\begin{document}

\begin{frontmatter}
\title{Large Deviations of Factor Models with Regularly-Varying Tails: Asymptotics and Efficient Estimation}

\begin{aug}
\author{\fnms{Farzad} \snm{Pourbabaee}*\ead[label=e1]{farzad@berkeley.edu}}
\and
\author{\fnms{Omid} \snm{Shams Solari}$^{\dagger}$
\ead[label=e3]{solari@berkeley.edu}
\ead[label=u1,url]{http://solari.stat.berkeley.edu}}\footnote[3]{Authors contributed equally to this manuscript.}


\affiliation{Department of Economics* and Department of Statistics$^{\dagger}$\\
University of California, Berkeley.}

\address{Department of Economics\\
University of California, Berkeley\\
\printead{e1}\\
\phantom{E-mail:\ }}

\address{University of California, Berkeley\\
Department of Statistics\\
\printead{e3}\\
\printead{u1}}
\end{aug}

\begin{abstract}
We analyze the \textit{Large Deviation Probability (LDP)} of linear factor models generated from non-identically distributed components with \textit{regularly-varying} tails, a large subclass of heavy tailed distributions. An efficient sampling method for LDP estimation of this class is introduced and theoretically shown to exponentially outperform the crude Monte-Carlo estimator, in terms of the coverage probability and the confidence interval's length. The theoretical results are empirically validated through stochastic simulations on independent non-identically Pareto distributed factors. The proposed estimator is available as part of a more comprehensive \texttt{Betta} package.
\end{abstract}



\end{frontmatter}

\section{Introduction}

\textit{Large deviation probability (LDP)} estimation is a well-studied problem in various branches of research; from finance and economics, to particle physics and weather forecasting. Researchers are often interested in the probability of occurrence of catastrophes, i.e. major over-shoots or under-shoots of an outcome comprised of a few input resources. A prominent example of which is the estimation of LDP for the factor models. The most well-studied case is estimating the LDP of sums of iid random variables. Namely, estimating $\BP\left[X_1+\ldots+X_N >x\right]$ for finite $N$ where $x$ is very large.

Such LDP estimation is well-studied when factors are thin-tailed and/or their class of distribution functions is stable under addition, e.g. Gaussian or Gamma factors.
The statistical analysis is particularly straightforward in these cases, because of the available closed-form expressions for the right or left tail probability. However, the majority of cases do not fall in this line, as in many cases this stability does not hold, and we can not appeal to analytical expressions for the deviation probability. An important example is the class of heavy-tailed distributions. Loosely speaking, for this class of random variables the rare events occur more frequently than in a light-tailed distribution such as Gaussian. \cite{gabaix2016power} enumerated many examples in which Power law distribution emerges, such as firm and city size, income and wealth distribution, and CEO compensations. \cite{asmussen2000rare} proposed the first efficient algorithm for LDP estimation of linear factor models with heavy-tailed iid components. They introduce a \textit{Conditional Monte-Carlo} (\textit{CMC}) algorithm which benefits from conditioning on order statistics. \cite{chan2011rare} utilize the same estimator of \cite{asmussen2000rare} in specific settings. They apply it to \textit{independent but non-identically distributed} (\textit{ind}) case, where the factors' distribution is restricted to be either Weibull or Pareto. Independent from \cite{chan2011rare}, we developed a CMC algorithm based on a comprehensive asymptotic description of how rare events occur in \textit{ind} setting when factors are \textit{regularly varying}. In contrast to \cite{chan2011rare} we provide theoretical guarantees establishing the faster convergence of our estimation algorithm relative to crude Monte-Carlo.

Conditioning is quite appealing since the classical Monte-Carlo methods for estimating LDP fall short, precisely because a large number of samples need to be drawn to get non-zero realizations of the sampling event. However, there are more efficient tools to address this problem, such as \textit{importance sampling}. The idea is essentially to sample from another probability measure that assigns more weight to the regions where the sampling function takes larger values, and then correct for the transformation of the sampling measure. \cite{ackerberg2000importance} showed that importance sampling can reduce the computational burden for smoothing the simulated moments, as first suggested by \cite{mcfadden1989method}.  

However, for the case of heavy-tail distributions, the general\footnote[1]{Referred as ``general" because many of the known methods of measure change are based on using the moment generating function as the Radon-Nikodym derivative. However, there are potentially heuristic ways to choose the sampling distribution according to the particular type of the unknown target variable, which is sought to be estimated.} measure transformation methods such as importance sampling are not favorable at best and inapplicable at worst. One reason is that higher moments as well as moment generating function, which are the essence of measure transformation methods, do not exist for this class. Secondly, due to the degeneracy of the likelihood ratio in high-dimensional models, these methods are not useful (\cite{rubinstein2016simulation}). In this paper, a novel technique (based on conditional Monte-Carlo sampling) is introduced to address the problem of tail estimation for the sum of \textit{ind} random variables belonging to a large subclass of heavy tails, namely \textit{regularly-varying} (RV) distributions \footnote[2]{This class of distributions is defined  in depth in \cite{feller2008introduction} section 8.8.}.

In the context of insurance risk, \cite{goovaerts2005tail} studies the tail asymptotics of randomly weighted sum of iid Pareto factors. Further, \cite{foss2010sums} find asymptotic results for the sum of conditionally independent factors under rather stringent conditions on the structure of factors' dependencies. \cite{albrecher2006tail} and \cite{kortschak2009asymptotic} use Copulas to capture the dependence structure of the factors, and derived similar asymptotic results for the tail probability. The main contributions of this paper are to provide asymptotics for the deviation probability of the sum and maximum of independent, $\BR$-valued, RV random variables; and to propose an improved Monte-Carlo method for estimating these likelihoods.

Likelihood estimation of such extreme events arises in many places: notably the extreme losses or profits of a portfolio exposed to multiple independent risk factors. Another example studied in \cite{acemoglu2017microeconomic} is the frequency of large economic downturns, and significant GDP departures from equilibrium trend, caused by the heavy-tail nature of micro shocks, wherein independent factors with Pareto tails add up and create large swings. The challenge is that for all these cases, the extreme tail probabilities are excessively small. Therefore, finding non-trivial confidence intervals for them is not just a matter of their size, but more importantly how big or small are they relative to the sought probability. Namely, it is the relative error, the length of the confidence interval divided by the point estimate, that matters for reporting the estimation precision. For example, in the case of a simple indicator random variable $1_A$, suppose that we are after $\mu=\BP A$. The per sample variance for the crude Monte-Carlo is $\mu(1-\mu)$, which indeed goes to zero as $\mu \to 0$. However, the relative error (standard deviation over mean), roughly scales as $1/\sqrt{\mu}$, which becomes arbitrary large. The proposed estimation method in this paper fixes this issue, which arises in the crude Monte-Carlo, and advances a bounded relative error as the size of the target probability vanishes.

The paper is organized as follows. In section \ref{gauss}, the optimality conditions for estimation are defined and some notions for the Gaussian case are explored. Next, in section \ref{rvfactors}, we establish some results on the tail asymptotics of RV sums, and the CMC algorithm, along with its concentration analysis and comparisons with crude Monte-Carlo. In section \ref{market}, the implications for portfolios of many assets with heavy tails are studied. In section \ref{sec: simulations}, the exponential efficiency of our proposed CMC algorithm relative to the crude Monte-Carlo estimator is demonstrated through the simulations. The proofs of the propositions and theorems along with simulation details are presented in the appendix.

\section{Gaussian Factor Model} \label{gauss}
In this section, we present a brief overview of the use of importance sampling as a method of variance reduction in the estimation of a large deviation probability under Gaussian factors. This would serve as an introduction that paves the way for the main results of the paper. Assume there are $M$ assets available in the market whose returns are driven by $k$ latent factors $\phi=(\phi_1,\ldots,\phi_k)$. The return to the $i$-th security is captured as a linear combination of the latent factors and the idiosyncratic risk, which is assumed uncorrelated with $\phi$:
\begin{equation}
\eta_i = \langle \beta_i,\phi\rangle +\ve_i
\end{equation}
Asset returns are all evaluated over the time interval $[t,t+\tau]$, where $\tau$ is the investment horizon. Observations of high-frequency data confirm that the distribution of returns deviates more intensely from Gaussianity as the investment horizon becomes shorter. For the moment, suppose that $\tau$ is long enough that we can assume Normal distributions both for the factors and asset specific risks, in particular assume $\phi \sim \mathcal{N}(0,I_k)$ and $\ve_i \sim \mathcal{N}(0,\sigma^2_i)$ are mutually independent. Let $\xi$ represent the return to market index, which is typically calculated as the market-cap weighted sum of security returns, but here for simplicity is taken as the unweighted average of $M$ returns:
\begin{equation}
\label{eqmarket}
\xi = \frac{1}{M} \sum_{i=1}^M \eta_i = \bar{\eta} = \langle \bar{\beta},\phi \rangle + \bar{\ve},
\end{equation}
which has the Normal distribution $\mathcal{N}\left(0,\norm{\bar{\beta}_2}^2 + \sum_{i=1}^M \sigma^2_i/M^2\right)$. \\
One can think of periods of market turmoil as the times when the market index reflects large downswings and upswings, namely $\lvert \xi \rvert > \lambda$, and one might want to estimate the probability of these large fluctuations, e.g $\BP\left[\xi > \lambda \right]$ for large $\lambda$. Since no closed form expression for this integral exists, we have to resort to simulation methods. However, crude Monte-Carlo sampling from the distribution of $\xi$ requires drawing a large number of samples to find some that surpass the threshold $\lambda$; importance sampling can help to reduce the required number of sample points, or alternatively reduce the variance of the point estimator. Given that the cumulative generating function $\psi(\theta)$ exists for Gaussian distribution for all $\theta \in \BR$, one possible choice to get an appropriate importance sampling distribution is the exponential measure change through
\begin{equation}
\psi(\theta) = \log \BE\left[e^{\theta\xi} \right] = \frac{\theta^2}{2}\left(\norm{\bar{\beta}}_2^2 + \frac{1}{M^2}\sum_{i=1}^M \sigma^2_i\right).
\end{equation}

Specifically, If $\BP$ denotes the actual probability measure for $\xi$, the exponentially twisted measure $\BP_\theta$ is then obtained by
\begin{equation}
\label{radon}
\frac{\mathrm{d} \BP_\theta}{\mathrm{d} \BP} = e^{\theta \xi -\psi(\theta)}.
\end{equation}

Now we can generate $n$ samples from $\BP_\theta$, and form the following sample average, which represents the unbiased estimator under the new measure $\BP_\theta$:
\begin{equation}
\frac{1}{n}\sum_{i=1}^n 1_{[\xi_i > \lambda]} \frac{\mathrm{d} \BP}{\mathrm{d}  \BP_\theta}(\xi_i)
\end{equation}

Denote the per-sample estimator by $Z(\lambda)= 1_{[\xi > \lambda]} \frac{\mathrm{d} \BP}{\mathrm{d} \BP_\theta}(\xi)$. The next definition spells out two notions of relative error.
\begin{definition}
\label{errordef}
The estimator $Z(\lambda)$ has \textit{bounded relative error} if 
\begin{equation}
\limsup_{\lambda \to \infty} \frac{\Var(Z(\lambda))}{\BE\left[Z(\lambda) \right]^2} <\infty,
\end{equation}
and is \textit{logarithmically efficient} (a weaker notion) if for some $\ve>0$
\begin{equation}
\limsup_{\lambda \to \infty} \frac{\Var(Z(\lambda))}{\BE\left[Z(\lambda) \right]^{2-\ve}} =0.
\end{equation}
\end{definition}
The following result, which is proved in \cite{asmussen2008applied}, sheds light on the efficiency of exponential twisting for a certain value of $\theta$. 
\begin{theorem}
The exponential change of measure in \eqref{radon} is logarithmically efficient for the unique parameter $\theta$ that solves $\lambda = \psi'(\theta)$.
\end{theorem}
As a result of this theorem, the optimal parameter for the measure change is 
\begin{equation}
\label{normalopt}
\theta^* = \frac{\lambda}{\norm{\bar{\beta}}_2^2 + \sum_{i=1}^M \sigma^2_i/M^2}.
\end{equation}
Having stated this theorem, the following lines summarize the simulation steps for the likelihood estimation of the market index large fluctuations in the Gaussian case:

\begin{enumerate}
\item Find $\theta^*$ from \eqref{normalopt}.
\item Draw random samples $\xi_i$, $i=1,\ldots,M$ from $\BP_{\theta^*}$.
\item Calculate $\frac{1}{n}\sum_{i=1}^n 1_{[\xi_i > \lambda]} e^{\psi(\theta^*)-\theta^* \xi_i}$ as an estimator of $\BP\left[ \xi>\lambda \right]$.
\end{enumerate}
As a result of twisting the sampling distribution, the relative error now scales as $\BP\left[ \xi > \lambda\right]^{-\ve/2}$, compared to $\BP\left[ \xi>\lambda\right]^{-1/2}$ for the classical Monte-Carlo. Equivalently, this boost shows us how to achieve a certain level of relative error with fewer sample points. However, this machinery can not always be employed, because the moment generating function need not always exist. Therefore, to find the optimal measure change we have to appeal to heuristic methods, or use other Monte-Carlo methods as explained further in the proceeding sections.


\section{Regularly-Varying Factors} \label{rvfactors}
In this section we study the consequences of dealing with independent factors with heavier tails than Gaussians. In particular, the factors are assumed to have regularly-varying tails, for example ones with Pareto tails. This class of distribution functions is contained in the larger family of sub-exponential distributions as defined below.
\begin{definition}
The distribution $F$ of a non-negative random variable $X$ is called sub-exponential, if
\[
\label{eq: subexpdef}
\lim_{x \to \infty}\frac{\BP\left[X_1+\ldots+X_N > x \right]}{\BP\left[X_1 >x \right]} = N ~ \text{for all} ~ N \geq 1,
\]
where $X_i$'s are iid copies drawn from $F$ \footnote{For more, check definition 1.3.3 in \cite{embrechts2013modelling}.}.
\end{definition}
This definition extends to probability distributions on the entire real line by restriction to the positive and negative halves. Then, the random variable $X\sim F$, taking values in $\BR$, is called sub-exponential if $X_+ = (X \vee 0)$ and  $X_- = -(X \wedge 0)$ are both sub-exponentials. Equation \eqref{eq: subexpdef} says that the probability that the sum of $N$ iid sub-exponential random variables exceeds a certain threshold is roughly $N$ times the probability that one of them exceeds that level. The question is thus what happens if the random variables are independent and individually sub-exponential but not necessarily identically distributed? Is the deviation probability for the sum related to the sum of deviation probabilities of the summands, and if so, under what conditions? As pointed out in the introduction, variations of these questions are studied under different conditions for the factors. 

In the remainder of this paper, we restrict ourselves to the case of sum of non-identical, independent, real-valued random variables. We answer this question under a mild condition, which is typically satisfied by long-tailed distributions.
\begin{condition} 
\label{hcond}Given the distribution $F$, there exists an eventually increasing function $h(x)$ such that $\lim_{x\to \infty} h(x) = \infty$ and \begin{equation}
\lim _{x\to \infty}\frac{\bar{F}(x+h(x))}{\bar{F}(x)}=1,
\end{equation}
where $\bar{F}(x):= 1-F(x)$.
\end{condition}
\begin{example} Suppose $X$ is Power law distributed with coefficient $\mu$, i.e $\BP\left[X> x \right] \propto x^{-\mu}$. Then, one can check that $h(x) = x ^{\delta}$, for any $0<\delta <1$, satisfies condition \ref{hcond}.
\end{example}
The following notation is used throughout the paper. 
\begin{notation}[Asymptotic equivalence] 
$f(x) \sim g(x)$ if $f(x)/g(x) \to 1$, as $x \to \infty$.
\end{notation}
The next definition expresses the notion of the RV distribution; see \cite{feller2008introduction} for more elaboration. 
\begin{definition}[Regularly-Varying (RV) distribution]
A distribution function $F$ has a \textit{regularly varying} tail, if $\bar{F}(x) \sim L(x)/x^{\alpha}$ as $x \to \infty$, where $\alpha>0$ and $L(\cdot)$ varies \textit{slowly} at infinity, i.e.
\begin{equation}
\lim_{x \to \infty} \frac{L(tx)}{L(x)}=1 ~ \text{for all} ~ t>0.
\end{equation}
\end{definition}
Functions such as $\log(x)$, $\log (\log(x))$ and any convergent function to a bounded level are examples of slow-variation.
The RV property depends only on the behavior of the distribution at infinity, so it does not matter how it behaves at intermediate points. One stylized observation about this family of distributions is that they  have finite moments of order less than $\alpha$, but not more. This will restrain us from using moment generating function to obtain large deviation results.
\begin{claim}
For all distribution functions of regular variation, we can take $h(x) = x^{\delta}$ with any $0<\delta <1$, and  condition \ref{hcond} will hold. The corollary of theorem 1 in section 8.8 of \cite{feller2008introduction} paves the way to prove this claim, which allows us to represent the slowly varying function $L(\cdot)$ as
\begin{equation}
L(x) = a(x) \exp\left(\int_1^x \frac{\ve(y)}{y} dy \right),
\end{equation}
where $\ve(x) \to 0$ and  $a(x) \to c$ as $x \to \infty$. Therefore,
\begin{equation}
\lim_{x \to \infty} \frac{L(x+x^\delta)}{L(x)} = \lim_{x\to \infty} \frac{a(x+x^\delta)}{a(x)} \lim_{x \to 
\infty}\exp\left(\int_x^ {x+x^\delta} \frac{\ve(y)}{y}dy \right),
\end{equation}
where the first term converges to 1, and the second term's exponent is approaching zero, because
\begin{equation}
\left| \int_x^ {x+x^\delta} \frac{\ve(y)}{y}dy \right|  \leq \frac{\sup_{y \in (x,x+x^\delta)} |\ve(y)| }{x^{1-\delta}} \to 0, ~~ \text{as} ~ x\to \infty.
\end{equation}
\end{claim}
\begin{remark}
	The next two results on large deviation of sum and maximum of a sequence of \textit{not} necessarily identical, independent and $\BR$-valued random variables, are built on the well-known properties of iid, $\BR^+$-valued RV random variables in \cite{embrechts2013modelling}\footnote{Precisely, for nonnegative, sub-exponential and iid random variables $(X_i)_{i=1,\ldots,N}$, $\BP\left[X_1+\ldots+X_N >x \right] \sim \BP\left[\max_{1\leq i \leq N} X_i >x\right]\sim N \BP\left[X_i>x\right]$ when $x \to \infty$; as stated in \cite{embrechts2013modelling} section 1.3.2.}.
\end{remark}
The following theorem assumes condition \ref{hcond}, and displays an asymptotic equivalence result for the tail probability of an independent RV sum.
\begin{theorem}
\label{tailequiv}
Suppose $X_1,\ldots,X_N$ are independent random variables in $\BR$, such that:
\begin{enumerate}[label=(\roman*)]
\item \label{ass1} An RV distribution $F$ exists, where $\bar{F}_i(x) \sim c_i \bar{F}(x)$ for all $i$'s and at least one $c_i \neq 0$,
\item \label{ass2} A function $h(\cdot)$ exists that satisfies condition \ref{hcond} for $F$,
\end{enumerate}
then the following asymptotic result holds:
\begin{equation}
\label{eq:tailequiv}
\BP\left[X_1+\ldots +X_N > x \right] \sim \sum_{i=1}^N \BP\left[X_i > x \right] \sim \left(\sum_{i=1}^N c_i \right) \bar{F}(x)
\end{equation}
\end{theorem}
Another interesting feature of sub-exponential distributions is the so-called \textit{catastrophe principle}, that roughly states that the iid sum of non-negative sub-exponential random variables is large if and only if one of them is large. To put it in a more precise way, here is the formal definition of this property:
\begin{definition}
\label{catasprop}
The distribution function $F$ with support on $[0,\infty)$ is said to satisfy the catastrophe principle, if
\begin{equation}
\BP\left[\max_{1 \leq i \leq N}X_i > x \right] \sim \BP\left[X_1+\ldots+X_N > x \right],~ \text{as}~ x \to \infty,
\end{equation}
where $X_1,\ldots,X_N$ are iid draws from $F$.
\end{definition}
In particular, the sub-exponential family has this property. However, we want to know what happens to the maximum factor under the more general conditions of theorem \ref{tailequiv}: when the random variables are independently drawn from non-identical distributions, and  can take negative as well as positive values. The next theorem examines the behavior of the maximum term up to a certain constant.
\begin{theorem}
\label{cat}
Suppose $X_1,\ldots,X_N$ are independently drawn from $F_1 \ldots, F_N $, and take values in $\BR$. Then, under the same conditions \ref{ass1} and \ref{ass2} of theorem \ref{tailequiv}, the following asymptotic result holds:
\begin{equation}
\begin{split}
\sum_{i=1}^N \BP\left[X_i >x \right]+ o(\bar{F}(x))& \leq \BP\left[\max_{1 \leq i \leq N}X_i > x \right] \\
&  \leq (1-e^{-1})^{-1}\sum_{i=1}^N \BP\left[X_i >x \right]+o(\bar{F}(x))
\end{split}
\end{equation}
\end{theorem}
\begin{remark}
There is nothing very special about the upper bound constant $(1-e^{-1})^{-1}$. It only paves the way for upper-bounding $e^{-x}$ by an affine function.  More is explained in the appendix, where we explain under a bit more stringent conditions, the exact statement of the catastrophe principle would be obtained, namely $\BP\left[\max_{1 \leq i \leq N}X_i > x \right] \sim\sum_{i=1}^N \BP\left[X_i >x \right]$ in this case.
\end{remark}
An important take-away from this result is that even under the extended case (non-identical and $\BR$-valued random variables), the catastrophe principle asymptotically holds up to a constant. More precisely, the probability that the sum exceeds a large value is of the \textit{same order} of the maximum summand exceeding the same threshold. This can also be interpreted in another sense: aggregate fluctuations do not become extremely large by accumulating small variations; rather, there has to be a single factor with large deviation to support such an extreme event.
\subsection{Conditional Monte-Carlo Algorithm}
The asymptotic result in theorem \ref{tailequiv} regarding the tail probability of the sum can be used to take $(\sum_{i=1}^N c_i) \bar{F}(x)$ as an estimator for $\BP\left[X_1+\ldots+X_N > x \right]$. However, this estimation performs weakly in many cases, and simulation based on that will be inaccurate. A conditional Monte-Carlo algorithm is developed in \cite{asmussen2006improved} to cope with the tail probability of sum of iid heavy tails. That idea is incorporated here to obtain an estimator for the sum of independent but non-identical factors. The algorithm goes as follows:
\begin{enumerate}[label=(\roman*)]
\label{enum: cmc}
\item Sample $X_i$ from its corresponding distribution $F_i$ for $i = 1, \ldots, N$.
\item Let $M_N = \max\{X_i: i\in [N]\}$.
\item Compute $Z(x)= \sum_{i=1}^N \BP\left[S_N > x, M_N = X_i \rvert X_{-i} \right]$
\end{enumerate}

The proposed $Z(x)$ is an unbiased estimator of $\BP\left[S_N>x\right]$\footnote{The proof of this claim is simple and thus omitted.}. The notation $X_{-i}$ is used to denote all random variables excluding $X_i$, and $S_N$ represents the sum of generated random variables from independent distributions, i.e $X_1+\ldots+X_N$. It is shown in \cite{asmussen2006improved} that the estimator in step 3 of the algorithm \ref{enum: cmc} has \textit{bounded relative error} for non-negative iid case, when the common distribution $F$ has RV form.
\begin{remark}
\label{rm:algoComplexity}
Consult appendix \ref{app:complexity} for a detailed discussion on the computational complexity of algorithm \ref{enum: cmc}. 
\end{remark}
The following theorem establishes the same result of \cite{asmussen2006improved}, but for the extended case of not necessarily identical $\BR$-valued factors.
\begin{theorem}
\label{regtail}
 If $F$ has regularly varying tail, then estimator $Z(x)$ in algorithm \ref{enum: cmc} has bounded relative error, namely
\begin{equation}
\limsup_{x \to \infty} \frac{\Var(Z(x))}{\BE[Z(x)]^2} < \infty.
\end{equation}
\end{theorem}
\begin{proof}[Proof of Theorem \ref{regtail}]  Denote $M_{N,-i} = \max \{X_{-i}\}$, $S_{N,-i} = \sum_{j\neq i}X_j$, and let $\widetilde{X_i}$ be an independent copy of $X_i$. Note that $Z(x)$ is implicitly a statistic generated from $X_1,\ldots,X_N$, thereby a random variable.
\begin{equation}
\begin{split}
Z(x) &= \sum_{i=1}^N \BP\left[S_N > x, M_N=X_i \rvert X_{-i} \right]= \sum_{i=1}^N \BP\left[\widetilde{X_i}> (x-S_{N,-i})\vee M_{N,-i} \rvert X_{-i} \right]\\
&= \sum_{i=1}^N \bar{F}_i\left((x-S_{N,-i})\vee M_{N,-i} \right)\sim \sum_{i=1}^N c_i\bar{F}\left((x-S_{N,-i})\vee M_{N,-i} \right)
\end{split}
\end{equation}
One can check that if $M_{N,-i} \leq x/N$ then $x-S_{N,-i} \geq x/N$, thereby $M_{N,-i} \vee \left(x-S_{N,-i}\right) \geq x/N$ always. Consequently, $Z(x)$ is asymptotically upper bounded by $\left(\sum_{i=1}^N c_i \right) \bar{F}(x/N)$, which yields to
\begin{equation}
\begin{split}
\limsup_{x\to \infty}\frac{\BE\left[Z(x)^2\right]}{\BE\left[Z(x)\right]^2} &\leq \lim_{x\to \infty} \frac{\left(\sum_{i=1}^N c_i \right)^2 \bar{F}(x/N)^2}{\left(\sum_{i=1}^N c_i \right)^2 \bar{F}(x)^2}= \frac{L^2(x/N)/(x/N)^{2\alpha}}{L(x)^2/x^{2\alpha}} = N^{2\alpha}.
\end{split}
\end{equation}
\end{proof}

\subsection{CMC Concentration and Efficiency Analysis}
The CMC algorithm can be repeated $n$ times with the outcome of $i$th step being referred as $Z_i(x)$, and the sample average is denoted by $\bar{Z}_n(x)$. Let $\mu(x) := \BP\left[X_1+\ldots+X_N>x \right]$, and $\sigma(x)^2 := \text{Var}(Z(x))$. Then, a simple application of central limit theorem yields to:
\begin{equation}
\frac{\bar{Z}_n(x)-\mu(x)}{\sigma(x)/\sqrt{n}} \stackrel{d}{\Longrightarrow} Z \stackrel{d}{=} \mathcal{N}(0,1)
\end{equation}
Therefore, one can get the following asymptotic confidence interval for the large deviation probability of $\bar{Z}_n(x)$:
\begin{equation}
\label{cltbound}
\begin{split}
\BP\left[\left|\bar{Z}_n(x)-\mu(x) \right| \leq \kappa \mu(x) \right] &= \BP\left[\left|Z\right| \leq \frac{\kappa\mu(x)}{\sigma(x)/\sqrt{n}} \right] + o_n(1)\\
& \geq \BP\left[ \left|Z \right| \leq \frac{\kappa\sqrt{n}}{ N^{\alpha}}\right] + o_x(1)+o_n(1)\\
& = \left(2\Phi\left(\frac{\kappa\sqrt{n}}{ N^{\alpha}}\right)-1\right) + o_x(1)+o_n(1),
\end{split}
\end{equation}
where $\Phi(\cdot)$ is the Gaussian CDF. Thus, for large enough $n$ and $x$ we have $\bar{Z}_n(x) \in \left(\mu(x)(1-\kappa),\mu(x)(1+\kappa) \right)$ with probability of at least $\left(2\Phi(\kappa \sqrt{n} N^{-\alpha})-1\right)$. Another way to find the concentration bound on $\bar{Z}_n(x)$ is to use the Markov's inequality:
\begin{equation}
\label{markovbound}
\BP\left[ \left| \bar{Z}_n(x) -\mu(x) \right| > \kappa \mu(x) \right] \leq \frac{\BE\left[\left( \bar{Z}_n(x) -\mu(x)\right)^2 \right]}{\kappa^2 \mu(x)^2} \leq \frac{N^{2\alpha}}{\kappa^2 n}+o_x(1),
\end{equation}
where the last inequality uses the final bound in theorem \ref{regtail} and holds for large enough $x$. Finally, we express a stronger approach to get a concentration bound based on the notion of sub-Gaussian random variables.
\begin{definition} [\cite{van1996weak}]
A random variable $X$ with mean $\mu = \BE X$ is called \textit{sub-Gaussian}, if there exists $\sigma >0$, such that
\begin{equation}
\BE\left[ e^{\lambda(X-\mu)}\right] \leq e^{\frac{\lambda^2 \sigma^2}{2}}, ~ \text{for all} ~ \lambda \in \BR.
\end{equation}
\end{definition}
\begin{remark}
\label{gaussdev}
Suppose that the random variable $X$ with mean $\mu$ is sub-Gaussian with parameter $\sigma$, then the following \textit{Chernoff} deviation bound would immediately fall out:
\begin{equation}
\BP\left[ \left|X-\mu \right| > t\right] \leq 2 e^{-t^2/2\sigma^2}
\end{equation}
\end{remark}
One can show that if $X$ takes value in $[a,b]$, then its sub-Gaussianity parameter is $(b-a)/2$. By looking at the computations in theorem \ref{regtail}, we can confirm that $Z(x) \in [0,\sum_{i=1}^N \bar{F}_i(x/N)]$, thereby $Z(x)$ is sub-Gaussian with parameter $\sum_{i=1}^N \bar{F}_i(x/N)/2$, and the following deviation bound results from remark \ref{gaussdev}:
\begin{equation}
\begin{split}
\BP\left[ \left|Z(x) - \mu(x) \right| > \kappa \mu(x)\right] &\leq 2\exp\left\{\frac{-2\kappa^2 \mu(x)^2}{\left( \sum_{i=1}^N \bar{F}_i(x/N)\right)^2}\right\}= e^{-2\kappa^2/ N^{2\alpha}}+o_x(1)
\end{split}
\end{equation}

In the last step we use the tail approximation for both $\mu(x)$ and the sum in the exponent's denominator. This is a one-shot bound, namely just for one trial of CMC algorithm, whereas if we repeat this process $n$ times, and take the sample average, then we get a much sharper precision:
\begin{equation}
\label{subbound}
\begin{split}
\BP\left[\left|\bar{Z}_n(x) - \mu(x) \right| > \kappa \mu(x) \right] & \leq 2\exp\left\{\frac{-2n\kappa^2 \mu(x)^2}{\left( \sum_{i=1}^N \bar{F}_i(x/N)\right)^2}\right\}= e^{-2n\kappa^2 /N^{2\alpha}}+o_x(1)
\end{split}
\end{equation}

As can be viewed in all three bounds \eqref{cltbound}, \eqref{markovbound} and \eqref{subbound} the ratio $N^{2\alpha}/n$ turns out to be the key parameter controlling the decay rate of error probability. For instance, if $N=10$, and $\alpha=2$, we need to repeat CMC algorithm $10^4$ times to get small error probability. An important observation here is that $n$ scales proportional to $N^{2\alpha}$, thus for fixed error rate smaller values of $\alpha$ lead to faster convergence rate, which makes more sense once we recall that the smaller levels of $\alpha$ correspond to the fatter tails. Therefore, the tail asymptotic equivalence relation will be achieved at smaller $x$'s, equivalently, the error in tail probability estimation would be smaller for fixed $x$.

The proposed CMC algorithm asymptotically outperforms the crude Monte-Carlo sampling in the sense of estimator's efficiency, namely for certain precision level $\kappa$, the deviation probability of CMC estimator is smaller than its regular sample mean counterpart, known as
\begin{equation}
\hat{\mu}_n(x):= \frac{1}{n} \sum_{k=1}^n 1_{\left[ X_1^{(k)}+\ldots + X_N^{(k)} > x\right]},
\end{equation}
where $X_i^{(k)}$ is the $k$th independent draw from $F_i$. The main theoretical result of the paper is presented next, in that we establish the exponential boost obtained via the proposed CMC estimator relative to the crude Monte-Carlo counterpart.

\begin{theorem}\label{cmceff}
For any precision level $0 < \kappa < 1$, the CMC estimator $\bar{Z}_n(x)$ is exponentially more efficient than $\hat{\mu}_n(x)$. Namely, for any $0<r<2\kappa^2 N^{-2\alpha}$ 
\begin{equation}
\label{relEff}
\limsup_{x \to \infty}\lim_{n \to \infty} \left\{ r n +\log\left( \frac{\BP\left[ \left|\bar{Z}_n(x) - \mu(x) \right| > \kappa \mu(x) \right]}{\BP\left[ \left|\hat{\mu}_n(x) - \mu(x) \right| > \kappa \mu(x) \right]}\right) \right\} \leq  0.
\end{equation}
\end{theorem}


\subsection{Importance Sampling Algorithm}
The goal in this part is to develop an alternative to CMC based on importance sampling. Inspired by the argument in previous part, we exploit the \textit{partitioning} method based on $M_N$. Suppose $f_i$ is the density of $X_i$, and $\tilde{f}_i$ is the alternative density, which is the candidate for importance sampling. Let $\mathrm{d} \BP_{(i)} = \mathrm{d} F_{-i} \otimes \mathrm{d} \widetilde{F}_i$ be the product measure generated from all original distributions bare $F_i$, where $\tilde{F}_i$ is used instead, and let $\BE_{(i)}$ express the expectation with respect to $\BP_{(i)}$. After all, the importance sampling steps follow as:
\begin{enumerate}
\item Generate $X_i \sim F_i$, and $\widetilde{X_i} \sim \widetilde{F}_i$.
\item Let $S_N^{(i)}:= S_{N,-i}+\widetilde{X}_i$, and $M_N^{(i)}=\max\{X_{-i},\widetilde{X}_i\}$.
\item Take $\sum_{i=1}^N \frac{f_i(\widetilde{X}_i)}{\tilde{f}_i(\widetilde{X}_i)} 1_{[S_N^{(i)}>x, M_N^{(i)} = \tilde{X}_i]}$ as an estimator for $\BP\left[S_N>x \right]$.
\end{enumerate}
To show the unbiasedness of the estimator, let us for example take the expectation of the $i$th summand with respect to $\BE_{(i)}$:
\begin{equation}
\begin{split}
\BE_{(i)}\left[\frac{f_i(\widetilde{X}_i)}{\tilde{f}_i(\widetilde{X}_i)}1_{[S_N^{(i)}>x, M_N^{(i)} = \tilde{X}_i]} \right] &= \int \frac{f_i(\tilde{x}_i)}{\tilde{f}_i(\tilde{x}_i)}1_{[S_N^{(i)}>x, M_N^{(i)} = \tilde{x}_i]} \left(\prod_{j\neq i}f_j(x_j)dx_j\right) \tilde{f}_i(\tilde{x}_i)d\tilde{x}_i \\
&= \int 1_{[S_N>x, M_N = x_i]}\prod_{j}f_j(x_j)dx_j\\
&=\BP\left[ S_N> x, M_N=X_i\right]
\end{split}
\end{equation}

Hence, $\sum_{i=1}^N\BP\left[ S_N> x, M_N=X_i\right] = \BP\left[S_N>x \right]$ and the unbiasedness is resulted.
Although the introduced estimator is unbiased, but as a downside it is shown in \cite{asmussen2006improved} that even for iid non-negative factors, it falls behind the CMC estimator let alone for our purpose. Moreover, one needs to find appropriate candidates for sampling distributions ($\tilde{f}_i$'s), where there is no general recipe to follow besides heuristics. The related literature is yet to find appropriate candidates for the sampling distributions in importance sampling, so likewise the question remains open in our case.


\section{Market Portfolio Large Deviation Probability}  \label{market}
One of the main motivations of studying RV distributions in this paper was to capture the large deviations of asset returns, as initially laid out for the Gaussian case. Now, consider the scenario in which the factor returns have Power law tails, i.e $\BP\left[\phi_i > x \right] \propto x^{-\tau_i}$, that happens to be the case in many empirical stock return observations, see for example \cite{cont2001empirical} and \cite{gopikrishnan1998inverse}. Then, the demeaned market index return can be modeled as the sum of independent zero mean factors combined with an independent noise, as seen before in \eqref{eqmarket}:
\begin{equation}
\xi = \sum_{i=1}^k \bar{\beta_i}\phi_i +\bar{\ve}.
\end{equation}
Since $\phi_i$ is assumed to have Power law tail, so does $\bar{\beta}_i\phi_i$ with the same tail coefficient. Therefore, letting $\tau = \min\{\tau_i: i \in [k]\}$ and $\gamma \in \{i: \tau_i=\tau\}$, the supporting distribution $F$ in the sense of theorem \ref{tailequiv} would be a Power law with coefficient $\tau$ (more precisely $F \stackrel{d}{=} \beta_\gamma \phi_\gamma$), and
\begin{equation}
c_i=\left\{ \begin{array}{lr}
\lim_{x\to \infty}\frac{\BP\left[\bar{\beta}_i \phi_i >x\right]}{\BP\left[\bar{\beta}_\alpha \phi_\alpha >x\right]} & \tau_i = \tau\\
0 & \tau_i > \tau
\end{array} \right.
\end{equation}

Moreover, hypothetically one can impose Gaussian structure on the idiosyncratic risk terms, and treat them as independent factors that have vanishing tail probabilities relative to the heaviest tail component, $\beta_\gamma \phi_\gamma$, namely \begin{equation}
\lim_{x\to \infty} \BP\left[\bar{\ve}>x \right]/\BP\left[\bar{\beta}_\gamma \phi_\gamma>x \right]=0.
\end{equation}

Then, the result of theorem \ref{tailequiv} implies that, for large $\lambda$:
\begin{equation}
\BP\left[\xi >\lambda \right] \sim \left(\sum_{i=1}^k c_i\right) \BP\left[ \bar{\beta}_\gamma \phi_\gamma>\lambda\right] \propto x^{-\tau_\gamma}
\end{equation}

As a result of this asymptotic tail equivalence, we can contemplate that only the factors with the heaviest tails contribute to the extreme events, and the market large fluctuations are mainly driven by them. Particularly, in terms of hedging against extreme events, the risk managers shall not worry about the factors with fat body distribution but light tails, even if they add a sizable portion of the portfolio variance, rather they should mainly concern about highly skewed ones.

Next, let us investigate the case, where the market portfolio is generated by aggregating a large number of individual stocks, uniformly weighted without loss of generality in this context. It is often observed that after factor extraction the remaining idiosyncratic parts reflect fat-tailed dispersions and treating them as Gaussians is quite unrealistic. Therefore, their deviation could possibly affect the aggregate index fluctuations. However, we show this is not true in the sense that each one can individually affect the fluctuations of its corresponding security, but once added together and averaged out, the aggregate noise deviation probability would have negligible effect compared to the contribution of factors with heavier tails. More precisely, as described above let $\eta_i = \langle \beta_i,\phi \rangle + \ve_i$ be the return to the $i$th security, while $\ve_i$ is no longer required to be Gaussian, but can take any RV form. The following proposition asserts this claim in a more definitive form.
\begin{proposition}
\label{marketthm}
Let $\eta_i = \langle \beta_i,\phi \rangle + \ve_i$ be the return to the $i$th security, such that idiosyncratic residuals likewise the factor returns are independent and have RV tails. Then, given the existence of a supporting distribution $F \sim L(x)/x^\alpha$ as in \ref{hcond} with $\alpha>1$, and uniformly bounded proportionality coefficients $\{c_i\}$ of individual noise distributions with respect to $F$ ($\max_{i \in [M]} c_i < c$), we get
\begin{equation}
\lim _{M \to \infty} \BP\left[ \frac{1}{M}\sum_{i=1}^M \ve_i > x\right] = 0,
\end{equation}
for fixed large $x$.
\end{proposition} 

The important result of this proposition is that under some regularity conditions on the residual security risks, the aggregate effect of these factors to the frequency of market index fluctuations will vanish for large portfolios of assets. Therefore, the large deviation of portfolios of many assets is mainly controlled by the common factors, which appear in all individual asset returns. One can think of this result as a version of the central limit theorem type argument across independent residuals, but in the case of independent and non-identical variables with fat tails. The market index large deviation probability can then be approximated as:
\begin{equation}
\begin{split}
\BP\left[\langle \bar{\beta}, \phi \rangle + \frac{1}{M} \sum_{i=1}^M \ve_i  >  x\right] &\sim \BP\left[\langle \bar{\beta}, \phi \rangle  > x \right] + \BP\left[ \frac{1}{M} \sum_{i=1}^M \ve_i  >  x \right]\\
&\stackrel{(*)}{\sim} \BP\left[\langle \beta, \phi \rangle  > x \right]
\end{split}
\end{equation}

The first asymptotic equivalence simply follows from theorem \ref{tailequiv} as $x$ gets large, and the second equivalence $(*)$ falls out by sending $M \to \infty$, in addition to the assumption that the average factor loading vector converges as $M \to \infty$, namely,
\begin{equation}
M^{-1}\sum_{i=1}^M \beta_i \to \beta.
\end{equation}

The methods such as CMC and importance sampling that introduced in previous section can now be employed to find estimators for extreme deviation probability of market index return.
\section{Simulations}
\label{sec: simulations}
Equation \ref{relEff} is perhaps the most consequential result of this read. In present section this result is unpacked and validated through several simulations. In what follows we demonstrate that our proposed estimator is exponentially more efficient than the crude Monte-Carlo estimator. More formally, we demonstrate that for any precision level $0 < \kappa < 1$ and finite number $N$ of independent Pareto factors, $$\log(\Lambda) = \log\left(\frac{\BP\left[ \left|\bar{Z}_n(x) - \mu(x) \right| > \kappa \mu(x) \right]}{ \BP\left[ \left|\hat{\mu}_n(x) - \mu(x) \right| > \kappa \mu(x) \right]}\right)$$ shrinks with a rate of at least $r$ as a function of the sample size $n$ as $x \rightarrow \infty$, where  $0<r<2\kappa^2 N^{-2\alpha}$ and $\alpha = \min_{1 \leq i \leq N} \alpha_i$ is the shape parameter corresponding to the factor with the heaviest tail. The minimum rate $r$ is estimated using a linear mixed-effects model, details of which is explained in Appendix \ref{subsection:mixed}. Also for brevity from now on we refer to $\log(\Lambda)$ as the LR ratio. 

\begin{remark}
\label{remark:r}
The CMC estimator, see algorithm \ref{enum: cmc}, is exponentially more efficient relative to the crude Monte-Carlo estimator with a rate of at least $r$ if there is an $r > 0$ for which the convex hull of $\left\{(i, \log(\Lambda_i)): i = 1, \ldots, n\right\}$ is bounded above by $f(i) = -ri$. 
\end{remark}

In what follows equation \ref{relEff} is validated through simulations while estimation sensitivity with respect to $\alpha$ and $x$ is studied.

\subsection{Variable Deviation Bound}
\label{subsection:varC}

Here we examine the relationship between the LR and the sample size as we move deviation bound further from the mean. Figure \ref{fig:varC} illustrates $\log(\Lambda)$ vs. the sample size $n$ as the deviation bound $x$ increases. It is clear from this result that our estimator maintains exponential efficiency through a wide range of deviation bounds and the rate of efficiency increases with that bound. Deviation bounds, along with their corresponding LDP and rate $r$ are presented in Table \ref{table:varC}.

\begin{figure}
    \centering
    \includegraphics[width = \textwidth]{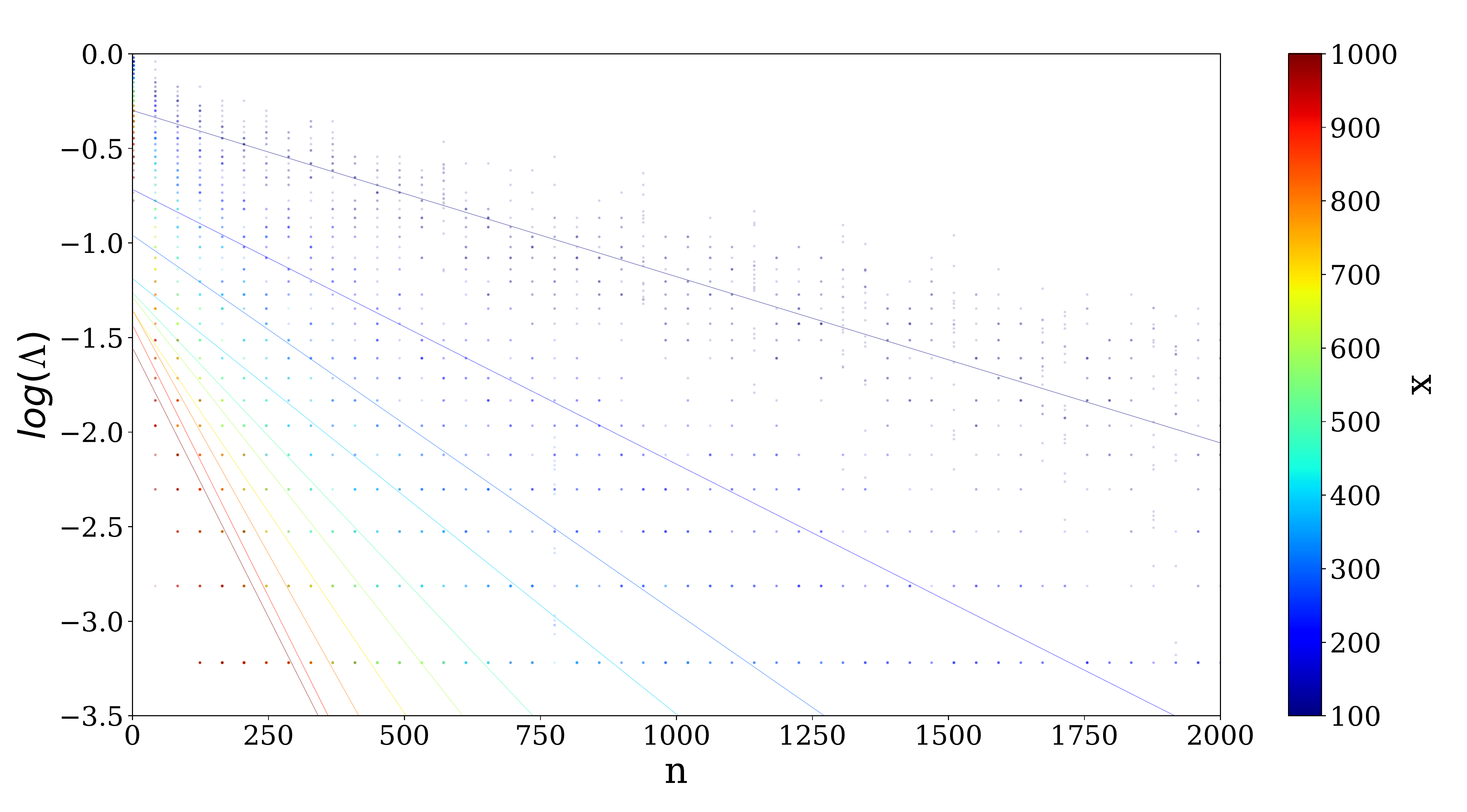}
    \caption{LR vs. sample size for increasing values of the deviation bound $x$. $r$ increases with $x$ which means that our CMC estimator performs exponentially better with increasing exponent as $\mu$ decreases, equiv. as $x$ increases.}
    \label{fig:varC}
\end{figure}

\begin{table}[h!]
\centering
\begin{tabular}{|c c c|} 
 \hline
 $x$ & $\mu$ & $r$\\ [0.5ex] 
 \hline\hline
 100& 1.921e-02& 8.7e-04\\
       200& 8.22e-03& 1.45e-03\\
       300& 5.14e-03& 1.99e-03\\
       400& 3.71e-03& 2.30e-03\\
       500& 2.89e-03& 3.03e-03\\
       600& 2.36e-03& 3.64e-03\\
       700& 1.99e-03& 4.25e-03\\
       800& 1.72e-03& 5.15e-03\\
       900& 1.51e-03& 5.74e-03\\
       1000& 1.35e-03& 5.69e-03\\
 \hline
\end{tabular}
\caption{$x, \mu$, and $r$ corresponding to simulation \ref{subsection:varC}}
\label{table:varC}
\end{table}

\subsection{Examining The Catastrophe Principle}
\label{subsec:catPr}
In this subsection two simulations are performed which aim to validate the Catastrophe principle. To this end, we simulate $M$ different factor models where each model contains $N$ Pareto factors with shape parameters
$\alpha_{i1}, \ldots, \alpha_{iN}$, $i = 1, \ldots, M$. We consider two cases: (1) groups that share the same $\alpha_{min}$, but the average tail thickness $\bar{\alpha}_i$ is different between models, (2) $\alpha_{min}$ is different between groups but each group shares the same $\bar{\alpha}$ with a group in the first case.

Figure \ref{fig:minConst} manifests the sensitivity of LR and $\mu$ to $\alpha$. Evidently, our CMC estimator maintains the exponential efficiency whose rate increases with mean tail thickness, denoted as $\bar{\alpha}$. 

\begin{figure}
    \centering
    \includegraphics[width = \textwidth]{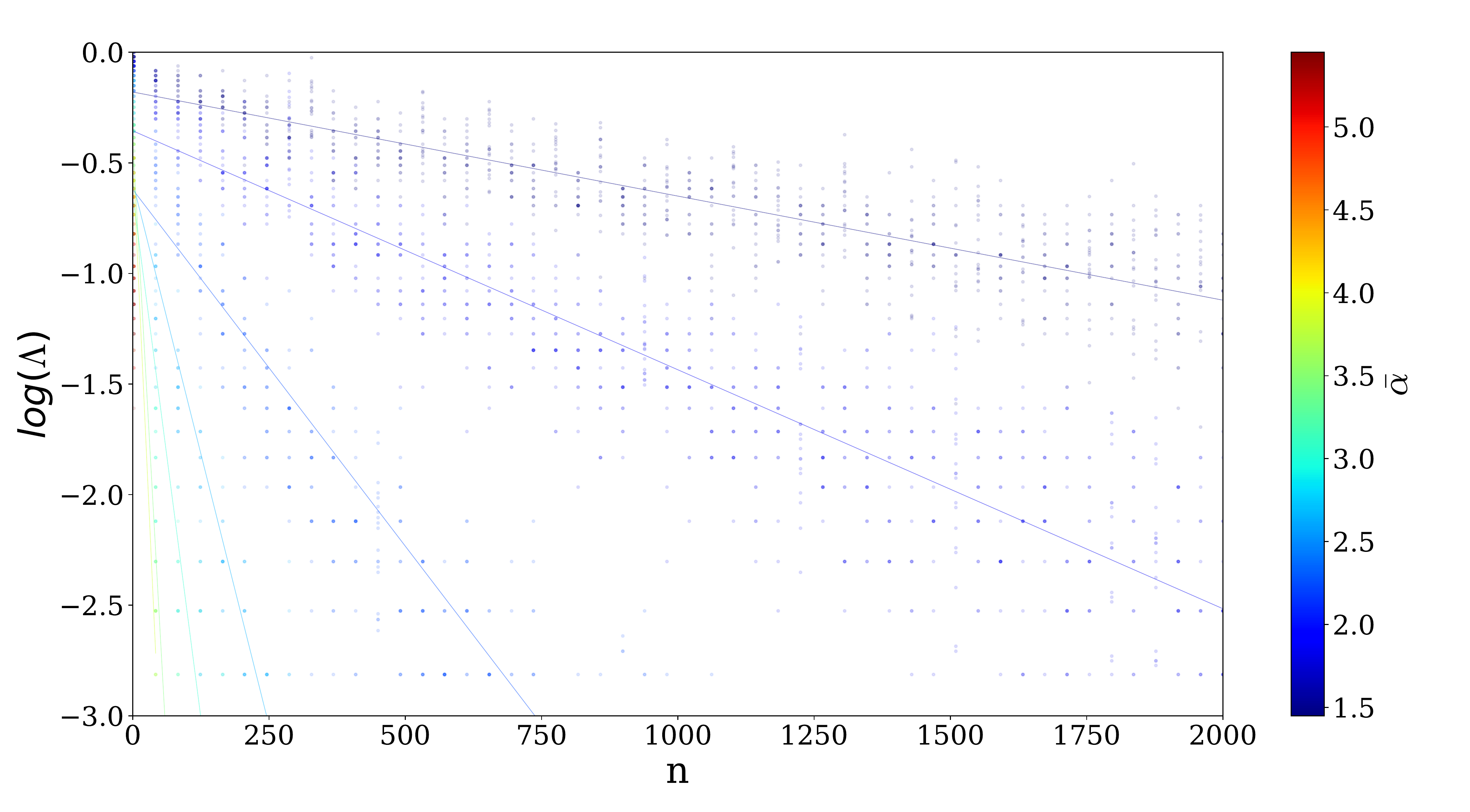}
    \caption{LR vs. sample size for various mean tail thicknesses while maximum tail thickness is constant. $r$ increases orders of magnitude with $\bar{\alpha}$.}
    \label{fig:minConst}
\end{figure}

\begin{table}[h!]
\centering
\begin{tabular}{|c c c|} 
 \hline
 $\bar{\alpha}$ & $\mu$ & $r$\\ [0.5ex] 
 \hline\hline
    1.45 & 3.367e-02& 4.7e-04\\
    1.85 & 1.389e-02& 1.08e-03\\
    2.25& 1.109e-02& 3.22e-03\\
    2.65& 1.058e-02& 9.82e-03\\
         3.05& 1.042e-02& 1.894e-02\\
         3.45& 1.034e-02& 4.432e-02\\
         3.85& 1.029e-02& 5.362e-02\\
         4.25& 1.025e-02& 2.6906e-01\\
         4.65& 1.022e-02& 3.0186e-01\\
         5.05& 1.020e-02& 4.1787e-01\\
         5.45& 1.018e-02& 4.4620e-01\\
 \hline
\end{tabular}
\caption{$\bar{\alpha}, \mu$, and $r$ while $\alpha_{min}$ is constant\textbf{}. LR increases by orders of magnitude while $\mu$ does not change significantly. This is an empirical validation of the Catastrophe principle, since it demonstrates that the LDP is driven by the heaviest tail and few smaller perturbations do not add up to a significant change in LDP.}
\label{table:minConst}
\end{table}

Table \ref{table:minConst} helps characterizing this increase more clearly. According to this table, $\mu$ is not sensitive to $\bar{\alpha}$ if all other factors have tails which are significantly thinner than $\alpha_{min}$. However, $r$ increases orders of magnitude which points to the high variability inherent in Monte-Carlo method.

In order to construct better characterization of our CMC method's sensitivity to maximum tail thickness, the previous simulation is repeated but this time $\alpha_{min}$ is not constant anymore but shape parameters are chosen in a way that for each model in the previous simulation, there exists a model in this simulation with equal $\bar{\alpha}$. In essence, mean tail thicknesses are similar in the two simulations. As before, CMC method maintains its dominance as $\alpha_{min}$ increases; however, upon consulting Table \ref{table:minVar}, we observe that, contrary to the previous simulation,  while $\mu$ decreases by orders of magnitude between factor models, $r$ does not change drastically.

\begin{figure}
    \centering
    \includegraphics[width = \textwidth]{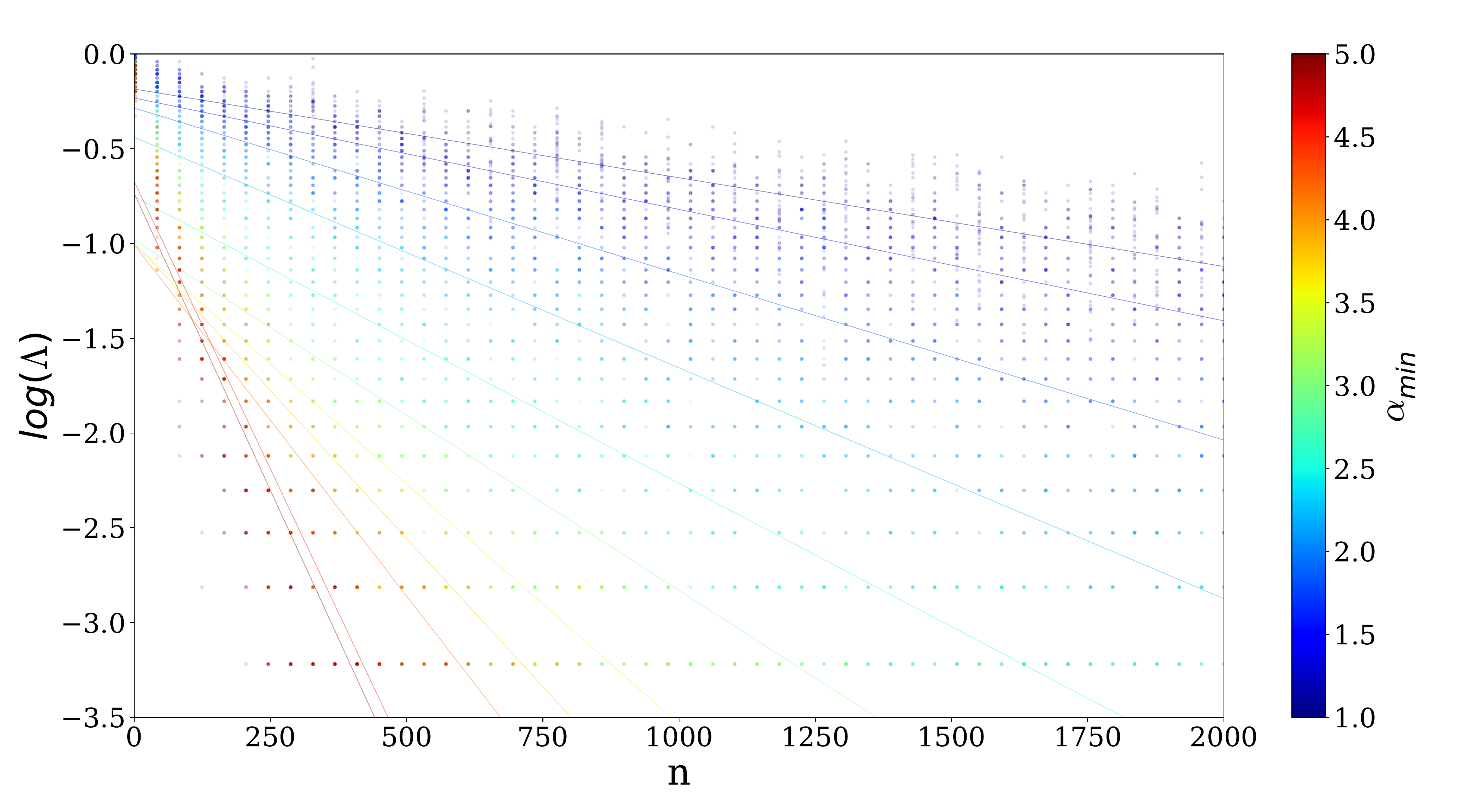}
    \caption{The same simulation as in Figure \ref{fig:minConst} is repeated. However, $\alpha_{min}$ is no longer constant.}
    \label{fig:minVar}
\end{figure}

\begin{table}[h!]
\centering
\begin{tabular}{|c c c|} 
 \hline
 $\alpha_{min}$ & $\mu$ & $r$\\ [0.5ex] 
 \hline\hline
  1.0 & 3.36e-02& 4.6e-04\\
        1.4 & 5.14e-03& 5.87e-04\\
        1.8 & 7.90e-04& 8.75e-04\\
        2.2 & 1.22e-04& 1.21e-03\\
        2.6 & 1.92e-05& 1.51e-03\\
        3.0 & 3.02e-06& 1.84e-03\\
        3.4 & 4.77e-07& 2.53e-03\\
        3.8 & 7.54e-08& 3.12e-03\\
        4.2 & 1.19e-08& 3.70e-03\\
        4.6 & 1.88e-09& 6.06e-03\\
        5.0 & 2.98e-10& 6.26e-03\\[0.5ex] 

 \hline
\end{tabular}
\caption{$\alpha_{min}, \mu$, and $r$, with $\bar{\alpha}$ kept equal between each row of this table and Table \ref{table:minConst}. Note the extreme change in $\mu$ while $r$ does not change significantly.}
\label{table:minVar}
\end{table}

\section{Conclusion}
\label{sec:conclusion}

This read covers a comprehensive asymptotic characterization of \textit{LDP}s in the case of linear factor models with \textit{ind Regularly Varying} factors. Exploiting this characterization, a Conditional Monte-Carlo estimator is proposed which has the same empirical time complexity, see Appendix \ref{app:complexity}, but is proven to be exponentially more efficient relative to the crude Monte-Carlo, see theorem \ref{relEff}. This claim was validated through extensive simulations while empirically characterizing the large deviation probabilities of the aforementioned factor models. Thus providing empirical support for the theoretical results presented in section \ref{rvfactors}. We hope to generalize the results of this article, especially the asymptotic behavior of linear factor models, to larger class of factor models, i.e. $\phi(\textbf{X})$ $ \mathbf{X} \in \mathbb{R}^N$ where $\phi \in \Phi$ a larger class of functions. Another important future direction is to study the LDP of factor models where $X_i$ are not necessarily independently distributed which can be used to estimate the LDP of portfolios with dependent assets.

\appendix

\section{Proofs}\label{app}

\subsection{Proof of Theorem \ref{tailequiv}}

First, the following lemma is proven, then the theorem's proof follows.
\begin{lemma}
\label{negl}
Let $F$ have regularly varying tail, namely $\bar{F}(x) \sim L(x) x^{-\alpha}$ for some $\alpha>0$. Then, there exists $0 < \delta < 1$, such that for $h(x)=x^\delta$,
\begin{equation}
\bar{F}(h(x))^2 = o(\bar{F}(x)).
\end{equation}

\end{lemma}

\begin{proof}[\unskip\nopunct]
Lemma 2 of chapter 8 in \cite{feller2008introduction} ensures that for every $\ve > 0$, there exists $x_0$, such that for all $x>x_0$: $x^{-\ve} < L(x) < x^\ve$. Now one can check that by taking $\ve < \alpha/5$ and $1>\delta > 3/4$ the desired result follows:
\begin{equation}
\frac{\bar{F}(h(x))^2}{\bar{F}(x)} \leq x^{\ve(2\delta+1)-\alpha(2\delta-1)} \to 0, ~ \text{as} ~ x\to \infty.
\end{equation}

\end{proof}

\begin{proof}[\unskip\nopunct] I justify equation \eqref{eq:tailequiv} for the case of two random variables, $X_1$ and $X_2$, then the general case will follow by a straight induction. The argument goes through a similar line of proof as in \cite{foss2010sums}, but I am going to leverage the independence to relax some of its necessary conditions. The idea is to upper and lower bound $\BP\left[ X_1+X_2>x\right]$ by $\BP\left[X_1>x \right]+\BP\left[X_2>x \right]$ with some vanishing approximation errors (that are approaching 0 as $x\to \infty$, faster than $\bar{F}(x)$, henceforth denoted by $o(\bar{F}(x))$). First, the upper-bound is verified:

\begin{equation}
\label{upperbound}
\begin{split}
\BP\left[ X_1+X_2 > x\right] \leq \BP\left[X_1>x-h(x) \right]+\BP\left[X_2>x-h(x) \right]+\\ \BP\left[h(x) < X_1 \leq x-h(x), X_2 > x-X_1 \right]
\end{split}
\end{equation}

The first two terms can be approximated by leveraging assumptions \ref{ass1} and \ref{ass2} of the theorem. For example:

\begin{equation}
\BP\left[X_1>x-h(x) \right] \sim c_1\bar{F}(x-h(x)) \sim c_1 \bar{F}(x) \sim \BP\left[X_1>x \right]
\end{equation}

where the first and last approximations hold because of \ref{ass1}, and the middle one is guaranteed by \ref{ass2} and condition \ref{hcond}.
 Furthermore, the last probability will be of order $\bar{F}(x)$ as $x \to \infty$:

\begin{equation}
\begin{split}
\BP\left[h(x) < X_1 \leq x-h(x), X_2 > x-X_1 \right]& \leq \BP\left[h(x) < X_1 \leq x-h(x) \right] \BP\left[X_2 > h(x) \right]\\
& \leq \BP\left[ X_1>h(x)\right]\BP\left[ X_2>h(x)\right]\\
&\sim c_1c_2 \bar{F}(h(x))^2,
\end{split}  
\end{equation}

where the last term is of order $o(\bar{F}(x))$ because of lemma \ref{negl}, hence is negligible compared to the first two terms in equation \eqref{upperbound}, that concludes the upper bound.
Next, the lower bounding goes as:

\begin{equation}
\label{lower}
\begin{split}
\BP\left[ X_1+X_2>x\right] &\geq \BP\left[X_1> x+h(x), X_2>-h(x) \right]+ \BP\left[X_2> x+h(x), X_1>-h(x) \right] \\
&\BP\left[X_1>x+h(x),X_2>x+h(x) \right],
\end{split}
\end{equation}

where each of the first two terms decouples, and again because of presumptions \ref{ass1} and \ref{ass2} of the theorem, the first one for instance can be approximated as

\begin{equation}
\begin{split}
\BP\left[X_1> x+h(x), X_2>-h(x) \right] &= \BP\left[X_1> x+h(x)\right]\BP\left[ X_2>-h(x)\right]\\
&\sim c_1\bar{F}(x+h(x)) \sim c_1\bar{F}(x) \sim \BP\left[X_1>x \right].
\end{split}
\end{equation}

Similar reasoning implies that the third term in \eqref{lower} is of order $o(\bar{F}(x))$, therefore vanishing compared to the first two terms in \eqref{lower}. The lower bound is now justified, hence the first approximation in equation \eqref{eq:tailequiv} is concluded. Finally, approximation of the sum of tail probabilities with $\bar{F}$ follows immediately as a result of the first presumption of the theorem.
\end{proof}

\subsection{Proof of Theorem \ref{cat}}
\begin{proof}[\unskip\nopunct] First, the lower bound is shown:

\begin{equation}
\begin{split}
\BP\left[\max_{1 \leq i \leq N}X_i > x \right] &= 1-\BP\left[\max_{1 \leq i \leq N}X_i \leq x \right]\\
&=1-\prod_{i=1}^N \left(1-\BP\left[X_i > x \right]\right)\\
&\geq 1-\prod_{i=1}^N \exp\left( -\BP\left[X_i > x \right]\right) \\
&= 1-\exp\left(-\sum_{i=1}^N\BP\left[X_i > x \right] \right)\\
&=\sum_{i=1}^N \BP\left[X_i >x \right]+o(\bar{F}(x)),
\end{split}
\end{equation}

where the last equality is an immediate application of the Taylor's lemma. Showing the upper bound mainly falls in the same steps, but requires invoking the inequality $e^{-x} \leq 1-(1-e^{-1})x$, that holds for $x\in [0,1]$.\\

\begin{equation}
\label{upbnd}
\begin{split}
1-\prod_{i=1}^N \left(1-\BP\left[X_i > x \right]\right) &\leq 1-\prod_{i=1}^N \exp\left(-(1-e^{-1})^{-1} \BP\left[X_i>x \right]\right)\\
&= 1-\exp\left(-(1-e^{-1})^{-1}\sum_{i=1}^N\BP\left[X_i>x \right] \right)\\
&=(1-e^{-1})^{-1}\sum_{i=1}^N \BP\left[X_i >x \right]+o(\bar{F}(x))
\end{split}
\end{equation}

Through a graphical scheme it becomes clear that $e^{-x} \leq 1-ax$ for $a<1$, and small enough $x$. Therefore, it is possible to approach $a \uparrow 1$ and control for the size of all $\bar{F}_i(x)$, $i=1,\ldots,N$. Under the case where the convergence of $\bar{F}_i(x)/c_i\bar{F}(x)$ (in condition \ref{ass1} of theorem \ref{tailequiv}) is uniform over all $i=1,\ldots,N$, one can send $a$ to 1 from below slower than the speed of $\bar{F}(x)\to 0$, thereby a tighter upper bound will be obtained in \eqref{upbnd} with the pre-factor $1$ rather than $(1-e^{-1})^{-1}$. 
\end{proof}

\subsection{Proof of Theorem \ref{cmceff}}
\begin{proof}[\unskip\nopunct] 
To prove the proposition we need the following lemma, that paves the way for the main verification.
\begin{lemma}
\label{kl}
Let $S_n:=\sum_{k=1}^n \xi_k$, where $\xi_k$'s are iid Bernoulli random variables with success probability of $\alpha$, then

\begin{equation}
\BP \left[ S_n \leq  n \delta \right] \geq \frac{1}{n+1} e^{-n D\left(\frac{\lfloor n\delta \rfloor }{n} || \alpha\right)},
\end{equation}

where $D(\cdot || \cdot)$ is the Kullback-Leibler divergence, that is known to be 

\begin{equation}
D(\delta || \alpha) = \delta \log\left( \frac{\delta}{\alpha} \right) +(1-\delta) \log \left( \frac{1-\delta}{1-\alpha} \right).
\end{equation}

\end{lemma}
\begin{proof}
For the notational simplicity let $m = \lfloor  n\delta \rfloor$, and $\tilde{\delta} = m/n$, then:

\begin{equation}
\label{binotail}
\begin{split}
\BP \left[ S_n \leq  n \delta \right] &= \sum_{k=0}^m \binom{n}{k} \alpha^k (1-\alpha)^{n-k}\\
&\geq \binom{n}{m} \alpha^m (1-\alpha)^{n-m} = \binom{n}{m} e^{n\left( \tilde{\delta} \log\alpha + (1-\tilde{\delta})\log(1-\alpha)\right)}
\end{split}
\end{equation}

Take the auxiliary binomial random variable $Y \sim \text{Bin}(n,\tilde{\delta})$, then $\BP\left[ Y=\ell\right]$ is maximized when $\ell =m= \lfloor n\delta \rfloor$. The following loose bound falls out for $\binom{n}{m}$:

\begin{equation}
\begin{split}
1 &= \sum_{\ell=0}^n \binom{n}{\ell} \tilde{\delta}^ \ell (1-\tilde{\delta})^{n-\ell}\\
&\leq (n+1) \binom{n}{m} \tilde{\delta}^m (1-\tilde{\delta})^{n-m}\\
&= (n+1)\binom{n}{m} e^{n\left( \tilde{\delta} \log\tilde{\delta} + (1-\tilde{\delta})\log(1-\tilde{\delta})\right)}
\end{split}
\end{equation}

Implying that $\binom{n}{m} \geq (n+1)^{-1} e^{-n\left( \tilde{\delta} \log\tilde{\delta} + (1-\tilde{\delta})\log(1-\tilde{\delta})\right)}$. Then, the proposed bound in the lemma drops out once this lower bound for $\binom{n}{m}$ is substituted in \eqref{binotail}.
\end{proof}
Now we can return to the proof of the proposition, first by finding the lower bound for deviation probability of $\hat{\mu}$:

\begin{equation}
\label{probdec}
\begin{split}
\BP\left[ |\hat{\mu}_n(x) - \mu(x)| > \kappa \mu(x)\right] = \BP\left[ \hat{\mu}_n(x) < (1-\kappa)\mu(x)\right]+\BP\left[ \hat{\mu}_n(x) > (1+\kappa)\mu(x)\right]
\end{split}
\end{equation}

The first term is lower bounded using the result of lemma \ref{kl} as:

\begin{equation}
\label{low}
\begin{split}
\BP\left[\hat{\mu}_n(x) < (1-\kappa)\mu(x) \right] &\geq \frac{1}{n+1}\exp\left\{-nD\left(\frac{\lfloor n(1-\kappa)\mu(x)\rfloor}{n}~||~ \mu(x) \right)\right\} \\ 
& \geq \frac{1}{n+1}\exp\left\{-nD\left((1-\kappa)\mu(x)-1/n ~||~ \mu(x)\right)\right\}
\end{split}
\end{equation}

In a same manner the second term in \eqref{probdec} is lower bounded, with this in mind that $1-\hat{\mu}_n(x)$ is the Binomial sample mean in its own turn, but with the different success probability of $1-\mu(x)$:

\begin{equation}
\label{up}
\begin{split}
 \BP\left[ \hat{\mu}_n(x) > (1+\kappa)\mu(x)\right] &= \BP\left[ 1-\hat{\mu}_n(x) < 1-(1+\kappa)\mu(x)\right]\\
&\geq \frac{1}{n+1} \exp\left\{ -nD\left(\frac{\lfloor n(1-(1+\kappa)\mu(x))\rfloor}{n}~ || ~ 1-\mu(x) \right)\right\}\\
&\geq \frac{1}{n+1} \exp\left\{ -nD\left(1-(1+\kappa)\mu(x)-1/n ~ || ~ 1-\mu(x) \right)\right\}   
\end{split}
\end{equation}

Denote the KL-divergences in the exponents of \eqref{low} and \eqref{up} with $D_1$ and $D_2$, respectively. Then, the convexity of $x \mapsto e^{-nx}$ implies:

\begin{equation}
\label{econv}
\begin{split}
\BP\left[ |\hat{\mu}_n(x) - \mu(x)| > \kappa \mu(x)\right] &\geq \frac{1}{n+1}\left(e^{-nD_1}+e^{-nD_2} \right) \\ &\geq \frac{2}{n+1}e^{-n(D_1+D_2)/2}
\end{split}
\end{equation}

Then it is left to simplify and find an upper bound for $D_1+D_2$, which is mainly carried out by leveraging the inequality: $ x \geq \log(1+x) $ for $x\in(-1,1)$.

\begin{equation}
\label{D1}
\begin{split}
D_1 &= \left( (1-\kappa)\mu(x)-1/n \right) \log\left(\frac{(1-\kappa)\mu(x)-1/n}{\mu(x)} \right)+\\
&\left(1-(1-\kappa)\mu(x)+1/n\right) \log\left(\frac{1-(1-\kappa)\mu(x)+1/n}{1-\mu(x)} \right) \\ 
&\leq  \left( (1-\kappa)\mu(x)-1/n \right) \left(-\kappa-\frac{1}{n\mu(x)} \right) + \left(1-(1-\kappa)\mu(x)+1/n\right)  \left(\frac{\kappa \mu(x)+1/n}{1-\mu(x)} \right),
\end{split}
\end{equation}

and

\begin{equation}
\label{D2}
\begin{split}
D_2 &= \left( 1-(1+\kappa)\mu(x)-1/n \right) \log\left( \frac{1-(1+\kappa)\mu(x)-1/n}{1-\mu(x)}\right)+\\
& \left( (1+\kappa)\mu(x)+1/n\right) \log\left(\frac{(1+\kappa)\mu(x)+1/n}{\mu(x)} \right)\\
&\leq \left( 1-(1+\kappa)\mu(x)-1/n \right) \left(-\frac{\kappa \mu(x)+1/n}{1-\mu(x)} \right) + \left( (1+\kappa)\mu(x)+1/n \right) \left(\kappa+\frac{1}{n\mu(x)} \right).
\end{split}
\end{equation}

Therefore, the following upper bound on $D_1+D_2$ falls out by adding up \eqref{D1}, and \eqref{D2}:

\begin{equation}
D_1+D_2 \leq 2 \frac{\left( \kappa \mu(x)+1/n\right)^2}{\mu(x)(1-\mu(x))}
\end{equation}

After substitution of this bound in \eqref{econv}, it follows

\begin{equation}
\BP\left[ |\hat{\mu}_n(x) - \mu(x)| > \kappa \mu(x)\right] \geq \frac{2}{n+1}\exp\left\{  \frac{-n\left(\kappa \mu(x)+1/n\right)^2}{\mu(x)(1-\mu(x))}\right\}.
\end{equation}

By using the upper bound on deviation probability of $\bar{Z}_n(x)$ in \eqref{subbound}, we can see that

\begin{equation}
\begin{split}
\log\left( \frac{\BP\left[ \left|\bar{Z}_n(x) - \mu(x) \right| > \kappa \mu(x) \right]}{\BP\left[ \left|\hat{\mu}_n(x) - \mu(x) \right| > \kappa \mu(x) \right]}\right)  &\leq (n+1)\exp \left\{ \frac{-2n\kappa^2 \mu(x)^2}{\left( \sum_{i=1}^N \bar{F}_i(x/N)\right)^2} \right. + \\ 
&\left. \frac{n\kappa^2 \mu(x)}{1-\mu(x)}+\frac{1}{n\mu(x)(1-\mu(x))} + \frac{2\kappa}{1-\mu(x)}\right\}\\
&\leq (n+1) \exp\left\{ -n\kappa^2\left(\frac{2}{N^{2\alpha}}-\mu(x)+o(\mu(x))\right)+\right. \\
&\left. 2\kappa + o_x(1) + \frac{1}{n\mu(x)}(1+o_x(1))\right\},
\end{split}
\end{equation}

where in the last equation, I used the large $x$ asymptotic. Now, for any rate $r$ smaller than $2\kappa^2 N^{-2\alpha}$, $x$ can be taken large enough, so that the ratio of deviation probabilities decays faster than $e^{-rn}$. Consequently, the ratio of CMC estimator deviation probability over its crude Monte-Carlo counterpart decays exponentially in $n$, pointing to the claim of theorem \ref{cmceff}.
\end{proof}

\subsection{Proof of Proposition \ref{marketthm}}
\begin{proof}[\unskip\nopunct]
The result of theorem \ref{tailequiv} can be employed again to asymptotically approximate the deviation sum with sum of deviations:

\begin{equation}
\label{noisedev}
\BP\left[ \sum_{i=1}^M \ve_i > Mx \right] \sim \sum_{i=1}^M \BP\left[\ve_i > M x \right] \sim \left( \sum_{i=1}^M c_i \right) \frac{L(Mx)}{(Mx)^\alpha},
\end{equation}

The proportionality coefficients $c_i$ are defined in the usual fashion: $c_i= \lim_{x \to \infty}\frac{\BP\left[\ve_i>x\right] }{\bar{F}(x)}$, and as stated in the theorem are uniformly bounded by a constant say $c$. Therefore, relation \eqref{noisedev} can be upperbounded as:

\begin{equation}
\left( \sum_{i=1}^M c_i \right) \frac{L(Mx)}{(Mx)^\alpha} \leq \frac{c L(Mx)}{M^{\alpha-1} x^{\alpha}} \longrightarrow 0, ~~ \text{as} ~ M\to \infty
\end{equation}

The last conclusion holds because $L(\cdot)$ is slowly varying by definition and grows at a slower rate than any polynomial growth of $M$ (knowing that it is assumed $\alpha>1$).
\end{proof}

\section{Algorithm Complexity}
\label{app:complexity} 
As claimed in \ref{rm:algoComplexity}, our proposed CMC algorithm has almost the same time complexity w.r.t the sample size $n$. Recall that the added complexity in our algorithm is only a result of evaluating a univariate distribution $N$ times, due to \ref{catasprop}, which does not scale with $n$ and these evaluations can be done in $O(1)$ time using fast methods or hash tables in case of distributions with sparse support.



\section{Simulations}
Simulations in this paper are carried out using \texttt{Betta} package addressed in \ref{suppA}. In this section, for the sake of reproducibility, the settings under which the simulations in Section \ref{sec: simulations} were carried out are explained in detail.

Some parameters were kept constant between the simulations. The number of factors in each model, $N$, was set to 10. Note that since an analytical solution for LDP estimation is not available in our case, we need to rely on stochastic simulations to estimate $\mu$ more accurately. After several experiments with our CMC estimator, we observed that as the sample size increases, estimation variance decreases and the mean estimate stays very close to the mean estimates using crude MC. Therefore, in order to estimate $\mu$, we used the CMC estimator but with a very large, $n = 1e7$ sample size. As for the precision parameter $\kappa$ we set it to $5e-3$ after many experiments. $\kappa$ should be small enough such that the difference between the estimators becomes more clear and large enough such that simulations do not end up with \textsc{Nan}'s due to occurrences of $\log(0)$, especially with the MC estimator.

\subsection{Simulation \ref{subsection:varC}}
In this simulation $x$ was changed from 100 to 1000 with steps of length 100. $\bm{\alpha} \in \BR^{10}$ was set to values scattered equidistantly between 1 and 3.

\subsection{Simulation \ref{subsec:catPr}}
Here for the simulation where the minimum shape parameter was different between models, Figure \ref{fig:minVar}, $\bm{\alpha} \in \BR^{10}$ was chosen equidistantly between $[e, e+1]$ for $e$ chosen from a grid of length 10 uniformly placed in $[1,5]$. The same instructions were used in the constant minimum shape parameter simulation of Figure \ref{fig:minConst}; however, for each model, the first element of $\bm{\alpha}$ was dropped and 1 was appended to the vector. Then the elements of the vector, except the first element, were modified to keep a constant mean $\alpha$ among the models.

\subsection{Estimating r}
\label{subsection:mixed}
In order to validate equation \ref{relEff}, it is sufficient to show that the convex hull of $\left(n_i, \log(\Lambda)_i\right)$ is bounded above by an affine function $f(n_i) = -rn_i$ where $r > 0$.

For each $n_i$, $\log(\Lambda)$ is evaluated 50 times, for all of which we use a single estimate for $\mu$. Therefore we used a linear mixed-effects model to estimate $r$. This way any grouping effect is considered as a random effect. Here is the model used:
\begin{equation}
    \label{eq:mixed}
    Y_{ij} = \beta_0 + \beta_1 X_{ij} + \gamma_{1i} X_{ij} + \epsilon_{ij}
\end{equation}

In that $i$ is the group corresponding to $n_i$. REML was used to estimate the coefficients.


\begin{supplement}

\sname{Supplement A}\label{suppA}
\stitle{\texttt{Betta} Package}
\slink[url]{https://github.com/osolari/betta}

\sdescription{\texttt{Betta} is a python package developed for the purposes of this paper. Upon installation of the package, objects of class \texttt{CMC} will be available. These objects may be input to the methods in the \texttt{relativeEfficiencyLib} module which contains the methods used for creating the results in section \ref{sec: simulations}.}
\end{supplement}
 
\clearpage

\bibliographystyle{imsart-nameyear}
\bibliography{ref}

\end{document}